\documentclass[12pt]{article}
\usepackage[english]{babel}
\usepackage{subcaption}
\usepackage{cite}
\usepackage{amsmath,amssymb,amsfonts,amsthm}
\usepackage{algorithmic}
\usepackage{mathtools}
\usepackage{graphicx}
\usepackage{subfig}
\usepackage{textcomp}
\usepackage{xcolor}

\theoremstyle{plain}
\newtheorem{theorem}{Theorem}
\newtheorem{lemma}[theorem]{Lemma}

\newtheorem{corollary}[theorem]{Corollary}

\theoremstyle{definition}
\newtheorem{example}{Example}
\newtheorem{definition}{Definition}

\newcommand{\dmid}{\mathrel{\|}}

\def\BibTeX{{\rm B\kern-.05em{\sc i\kern-.025em b}\kern-.08em
    T\kern-.1667em\lower.7ex\hbox{E}\kern-.125emX}}


\title{Walk Matrix–Based Upper Bounds on Generalized Cospectral Mates}

\author{\small Muhammad Raza$^{{\rm a}}$\thanks{Corresponding author: mraza@itu.edu.pk}\quad\quad Mudassir Shabbir$^{\rm a}$\quad\quad Waseem Abbas$^{{\rm b}}$
\\
{\footnotesize$^{\rm a}$Department of Computer Science, Information Technology University, Lahore, Pakistan}\\
{\footnotesize$^{\rm b}$Department of Systems Engineering, The University of Texas at Dallas, Richardson, TX
}
}
\date{}

\begin{document}

\maketitle

\begin{abstract}
The problem of characterizing graphs determined by their spectrum (DS) or generalized spectrum (DGS) has been a longstanding topic of interest in spectral graph theory, originating from questions in chemistry and mathematical physics. While previous studies primarily focus on identifying whether a graph is DGS, we address a related yet distinct question: how many non-isomorphic generalized cospectral mates a graph can have? Building upon recent advances that connect this question to the properties of the walk matrix, we introduce a broad family of graphs and establish an explicit upper bound on the number of non-isomorphic generalized cospectral mates they can have. This bound is determined by the arithmetic structure of the determinant of the walk matrix, offering a refined criterion for quantifying the multiplicity of generalized cospectral graphs. This result sheds new light on the structure of generalized cospectral graphs and provides a refined arithmetic criterion for bounding their multiplicity.


\end{abstract}

\section{Introduction}
\label{section:intro}
The spectrum of a graph encodes significant combinatorial information and has long served as a powerful tool in addressing various problems in graph theory, even when those problems are not explicitly spectral in nature. A central question in spectral graph theory is: “Which graphs are determined by their spectrum (DS)?” This problem, which originated over 50 years ago in the context of chemistry, has attracted considerable attention in recent years.

In 1956, Günthard and Primas \cite{gunthard1956zusammenhang} posed this question in a paper connecting graph spectra to Hückel’s theory in chemistry. Numerous constructions of cospectral graphs—graphs sharing the same spectrum—have since been studied. For instance, Godsil and McKay \cite{godsil1982constructing} introduced the GM-switching method, which generates many pairs of cospectral graphs (including those with cospectral complements). Another notable result by Schwenk \cite{schwenk1973almost} showed that almost all trees are not DS.

This problem is also closely related to Kac’s famous question \cite{kac1966can}: “Can one hear the shape of a drum?” Fisher \cite{fisher1966hearing} modeled the shape of a drum using a graph, with the drum’s sound characterized by the eigenvalues of the graph. Thus, Kac’s question is essentially equivalent to the DS problem for graphs.

Formally, let \( G = (V, E) \) be a simple graph with $n$ vertices, meaning it is finite, undirected, and contains no loops or multiple edges. The \emph{adjacency matrix} \( A(G) \) is the $n \times n$ symmetric matrix defined by $A_{ij} = 1$ if vertices $v_i$ and $v_j$ are adjacent, and $A_{ij} = 0$ otherwise. The \emph{complement} of a graph $G$, denoted $\overline{G}$, is the graph on the same vertex set where two vertices are adjacent if and only if they are not adjacent in $G$. The adjacency matrix of the complement is given by $A(\overline{G}) = J - I - A(G)$, where $J$ is the $n \times n$ all-ones matrix and $I$ is the identity matrix.

Two graphs $G$ and $H$ are \emph{isomorphic}, denoted $G \cong H$, if there exists a bijection $f: V(G) \to V(H)$ such that $u$ and $v$ are adjacent in $G$ if and only if $f(u)$ and $f(v)$ are adjacent in $H$. In other words, isomorphic graphs have identical structure, differing only by a relabeling of vertices. The \emph{spectrum} of an undirected graph $G$, denoted $\operatorname{Spec}(G)$, is the multiset of eigenvalues of its adjacency matrix $A(G)$, counted with algebraic multiplicity. A graph $G$ is said to be \emph{determined by its spectrum} (DS) if every graph with the same spectrum as $G$ is isomorphic to $G$.

A widely studied extension of the above problem involves the notion of the \emph{generalized spectrum}, defined as the pair $(\operatorname{Spec}(G), \operatorname{Spec}(\overline{G}))$. Two or more non-isomorphic graphs are called \emph{generalized cospectral mates} if they share the same generalized spectrum. A graph \( G \) is said to be \emph{determined by its generalized spectrum} (DGS) if every graph with the same generalized spectrum as \( G \) is isomorphic to \( G \).

Wang et al.~\cite{wang2006sufficient} provided a sufficient condition for a family of graphs to be DGS, and subsequently proposed an exclusion algorithm for testing the DGS property~\cite{wang2006excluding}. Later, Wang~\cite{wang2013generalized} introduced a new family of graphs and showed that almost all graphs in this family are DGS, and later proved that all graphs in this family are DGS~\cite{wang2017simple}. Qiu et al.~\cite{qiu2021oriented} extended these ideas to oriented graphs, establishing an arithmetic criterion for a graph to be determined by its generalized skew spectrum. In further work, Qiu et al.~\cite{qiu2023generalized} provided an arithmetic criterion for regular graphs determined by their generalized $\Phi$-spectrum.

Recent advances have leveraged the properties of the walk matrix to address these questions. For a graph $G$ with $n$ vertices, the \emph{walk matrix} $W(G)$ is the $n \times n$ matrix whose columns are $e,\, A(G)e,\, \dots,\, A(G)^{n-1}e$, where $e$ is the all-ones column vector of length $n$. A graph $G$ is called \emph{controllable} if $W(G)$ is non-singular.

While much of the existing literature has focused on identifying whether a graph is DGS, our work takes a different direction: we establish an explicit upper bound on the number of non-isomorphic generalized cospectral mates a graph can have. This bound is determined by the number of distinct odd prime factors of the determinant of its walk matrix.

We denote by \( \mathbb{F}_p \) the finite field with \( p \) elements, and by \( \operatorname{rank}_p(M) \) the rank of an integral matrix \( M \) over \( \mathbb{F}_p \). Wang et al.~\cite{wang2023graphs} introduced a broad family of graphs, denoted \( \mathcal{H}_n \), of order \( n \), characterized by the property that for every \( G \in \mathcal{H}_n \), \( 2^{-\lfloor n/2 \rfloor} \det W(G) = p^2 b \), where \( p \) is an odd prime, \( b \) is an odd square-free integer, and \( \operatorname{rank}_p W(G) = n-1 \). For this family, they established the following theorem:

\begin{theorem}
\label{theorem:wang}
\cite{wang2023graphs}
Let \( G \in \mathcal{H}_n \). Then \( G \) has at most one non-isomorphic generalized cospectral mate; that is, \( G \) is either DGS or has a unique generalized cospectral mate.
\end{theorem}

Extending these results, we introduce a broader family of graphs, denoted \(\mathcal{F}_n\), which strictly contains \(\mathcal{H}_n\). Specifically, \(\mathcal{F}_n\) consists of all graphs \(G\) of order \(n\) such that \(2^{-\lfloor n/2 \rfloor} \det W(G)\) is odd and cube-free, and \(\operatorname{rank}_p W(G) = n-1\) for every odd prime \(p\) dividing \(\det W(G)\). Our main theorem establishes an explicit upper bound on the number of non-isomorphic generalized cospectral mates for graphs in this family:

\begin{theorem}
    \label{theorem:main}
    Let \( G \in \mathcal{F}_n \). Then \( G \) has at most \( 2^k - 1 \) non-isomorphic generalized cospectral mates, where \( k \) is the number of distinct odd primes whose square divides \( \det W(G) \).
\end{theorem}

The subsequent sections lay the groundwork for our main results. In Section~\ref{section:prelim}, we review essential background concepts and technical lemmas from the literature that underpin our approach. Section~\ref{section:primitive} introduces the notion of primitive matrices and explores their structural properties, which are crucial for bounding the number of generalized cospectral mates. Section~\ref{section:proof} presents the proof of Theorem~\ref{theorem:main}, detailing the arithmetic arguments and matrix-theoretic techniques involved. We then provide illustrative examples in Section~\ref{section:examples} to demonstrate the application of our results. Section~\ref{section:conclusion} concludes the paper with a summary of our findings and potential directions for future research.

\section{Preliminaries}
\label{section:prelim}
In this section, we present key results from the literature that will be utilized to prove Theorem ~\ref{theorem:main}. A matrix $Q$ is called regular if $Qe=e$ where $e$ is the all-ones matrix. The following theorem offers a simple characterization of when two oriented graphs have the same generalized spectrum:


\begin{theorem}
\label{theorem:q}
\cite{wang2006sufficient} Let \( G \) be graph. Then there exists a graph \( H \) such that \( G \) and \( H \) have the same generalized spectrum if and only if there exists a regular rational orthogonal matrix \( Q \) such that:
\begin{equation}
\label{eq:q}
Q^T A(G) Q = A(H),
\end{equation}
Moreover, if $G$ is controllable with respect to adjacency matrix, $Q$ is unique and $Q=W(G) W(H)^{-1}$.
\end{theorem}

Based on this theorem, we define the following set:
\begin{equation*}
\Gamma(G) = \{Q \in O_n(\mathbb{Q}) \mid Q^T A(G) Q = A(H) \text{ for graph } H\},
\end{equation*}
where \(O_n(\mathbb{Q})\) is the set of all regular orthogonal matrices with rational entries. We define the level of a matrix as follows:

\begin{definition}
Let \(Q\) be an orthogonal matrix with rational entries. The \emph{level} of \(Q\), denoted by \(\ell(Q)\) (or simply \(\ell\)), is the smallest positive integer \(x\) such that \(xQ\) is an integral matrix.
\end{definition}

The concept of the level of a matrix is central to understanding the spectral characterization of graphs. Specifically, if the level of a regular rational orthogonal matrix $Q$ is one, then $Q$ must be a permutation matrix. In this case, if $Q^T A(G) Q = A(H)$ for graphs $G$ and $H$, it follows that $G$ and $H$ are isomorphic.

Another important tool in this context is the Smith Normal Form (SNF) of integral matrices. An integral matrix \(V\) of order \(n\) is called \emph{unimodular} if \(\det{V} = \pm 1\). For any full-rank integral matrix \(M\), there exist unimodular matrices \(V_1\) and \(V_2\) such that \(M = V_1 N V_2\), where \(N = \operatorname{diag}(d_1(M), d_2(M), \ldots, d_n(M))\) is the SNF of \(M\), and the invariant factors satisfy \(d_i(M) \mid d_{i+1}(M)\) for \(i = 1, 2, \ldots, n-1\). The determinant of $M$ can thus be written as
\begin{equation}
    \det{M} = \pm \prod_{i=1}^{n} d_i(M).
\end{equation}

For a prime $p$, the rank of $M$ over $\mathbb{F}_p$ is equal to the number of invariant factors of $M$ not divisible by $p$. This relationship is useful for analyzing the structure of $M$ modulo $p$. The following lemma establishes a connection between the level of a rational orthogonal matrix and the last invariant factor of an associated integral matrix.

\begin{lemma}
\label{lemma:level-dn}
\cite{qiu2023smith}
Let $X$ and $Y$ be two non-singular integral matrices such that $QX = Y$, where $Q$ is a rational orthogonal matrix. Then $\ell(Q)$ divides the greatest common divisor of $d_n(X)$ and $d_n(Y)$, where $d_n(X)$ and $d_n(Y)$ are the $n$-th invariant factors of $X$ and $Y$, respectively.
\end{lemma}

Building on these observations, we note that for every \( Q \in \Gamma(G) \), the level \(\ell(Q)\) must divide \(d_n(W(G))\), the last invariant factor of the walk matrix \(W(G)\). However, certain arithmetic properties of \(W(G)\) can restrict which divisors of \(d_n(W(G))\) are actually attainable as levels for matrices in \(\Gamma(G)\). Specifically, for nonzero integers \(n\) and \(m\) and a positive integer \(k\), we write \(m^k \dmid n\) to indicate that \(m^k\) divides \(n\) exactly; that is, \(m^k \mid n\) but \(m^{k+1} \nmid n\).

\begin{lemma}
\label{lemma:2-out}
\cite{wang2013generalized}
Let $G$ be a graph and let $Q \in \Gamma(G)$. If $2^{\lfloor n/2 \rfloor} \dmid \det{W(G)}$, then $\ell(Q)$ is odd.
\end{lemma}

\begin{lemma}
\label{lemma:pk-out}
\cite{qiu2023smith}
Let $G$ be a graph, let \( Q \in \Gamma(G) \) and let \( p \) be an odd prime. If $p^k \dmid \det{W(G)}$ and $\operatorname{rank}_p(W(G)) = n-1$, then \( p^{k} \nmid \ell(Q) \) or equivalently, \( \ell(Q) = \frac{d_n(W(G))}{p} \).
\end{lemma}

\begin{corollary}
\label{corollary:l-struct}
Let $G\in \mathcal{F}_n$ and let \( Q \in \Gamma(G) \). Then $\ell(Q)$ is odd and square-free.
\end{corollary}

In the next section, we discuss the properties of matrices in $\Gamma(G)$.

\section{Primitive Matrices}
\label{section:primitive}

\begin{definition}
A regular rational orthogonal matrix $Q$ is called a \emph{primitive matrix} if $\ell(Q)$ is odd and for every prime $p$ dividing $\ell(Q)$, we have $\operatorname{rank}_p(\ell(Q)Q) = 1$.
\end{definition}

We will show in the next section that every matrix in $\Gamma(G)$ is primitive. We now establish several key properties of primitive matrices that are essential for our main results. The following lemmas are inspired by \cite{wang2023graphs,raza2025upper}.

\begin{lemma}
\label{lemma:level-mod}
Let \( Q \) be a primitive matrix, and let $x$ and $k$ be positive integers. If \( xQ \) is an integral matrix and \( xQ \equiv 0 \pmod{k} \), then \( \ell(Q) \mid \frac{x}{k} \).
\end{lemma}

\begin{proof}
Since \( xQ \equiv 0 \pmod{k} \), every entry of \( xQ \) is divisible by \( k \), so \( \frac{x}{k}Q \) is integral. By the definition of the level \( \ell(Q) \), which is the minimal positive integer such that \( \ell(Q)Q \) is integral, it follows that \( \ell(Q) \) divides \( \frac{x}{k} \). This completes the proof.
\end{proof}

\begin{lemma}
\label{lemma:uv}
Let \( u \) and \( v \) be $n$-dimensional integer column vectors, and let $p$ be an odd prime. Suppose:
\begin{enumerate}
    \item \( u \not\equiv 0 \pmod{p} \) and \( v \not\equiv 0 \pmod{p} \),
    \item \( u \) and \( v \) are linearly dependent over \( \mathbb{F}_p \),
    \item \( u^T u = v^T v \equiv 0 \pmod{p^2} \),
\end{enumerate}
then \( u^T v \equiv 0 \pmod{p^2} \).
\end{lemma}

\begin{proof}
If \( u = \pm v \), then \( u^T v = \pm u^T u \equiv 0 \pmod{p^2} \), so the result holds. Otherwise, since \( u \) and \( v \) are linearly dependent over \( \mathbb{F}_p \), there exist integers \( a \) and \( b \), not both zero modulo $p$, such that
\[
au + bv \equiv 0 \pmod{p}.
\]
Neither \( a \) nor \( b \) can be zero modulo $p$, since otherwise \( u \) or \( v \) would be zero modulo $p$, contradicting the assumptions.

Taking the inner product of both sides with themselves gives
\[
(au + bv)^T (au + bv) \equiv 0 \pmod{p^2},
\]
which expands to
\[
a^2 u^T u + 2ab u^T v + b^2 v^T v \equiv 0 \pmod{p^2}.
\]
Since \( u^T u \equiv v^T v \equiv 0 \pmod{p^2} \), this reduces to
\[
2ab\, u^T v \equiv 0 \pmod{p^2}.
\]
Because $2ab$ is not divisible by $p$, it follows that \( u^T v \equiv 0 \pmod{p^2} \).
\end{proof}

\begin{lemma}
\label{lemma:p-out}
Let $Q_1$ and $Q_2$ be primitive matrices and $p$ be an odd prime. Let \(\bar{Q}_1 = \ell(Q_1)Q_1\) and \(\bar{Q}_2 = \ell(Q_2)Q_2\). Suppose:
\begin{enumerate}
    \item \( p \) divides both \( \ell(Q_1) \) and \( \ell(Q_2) \),
    \item The column vectors of $\bar{Q}_1$ and $\bar{Q}_2$ share the same linear basis over $\mathbb{F}_p$,
    \item Both $\ell(Q_1)$ and $\ell(Q_2)$ are square-free.
\end{enumerate}
Then $p \nmid \ell(Q_1^T Q_2)$.
\end{lemma}

\begin{proof}
Let \(O = \bar{Q}_1^T \bar{Q}_2\), where each entry \(o_{ij} = u_i^T v_j\), with \(u_i\) and \(v_j\) the $i$th and $j$th columns of $\bar{Q}_1$ and $\bar{Q}_2$. Note that $Q_1^T Q_2$ is rational, while $\bar{Q}_1$, $\bar{Q}_2$, and $O$ are integral.

We claim $O \equiv 0 \pmod{p^2}$. Since $Q_1$ and $Q_2$ are orthogonal, $\bar{Q}_1^T \bar{Q}_1 \equiv 0 \pmod{p^2}$ and $\bar{Q}_2^T \bar{Q}_2 \equiv 0 \pmod{p^2}$, so $u_i^T u_i \equiv 0 \pmod{p^2}$ and $v_j^T v_j \equiv 0 \pmod{p^2}$. For any $i,j$, $u_i$ and $v_j$ are linearly dependent over $\mathbb{F}_p$.

Consider three cases:
\begin{itemize}
    \item If $u_i \not\equiv 0 \pmod{p}$ and $v_j \not\equiv 0 \pmod{p}$, Lemma~\ref{lemma:uv} gives $u_i^T v_j \equiv 0 \pmod{p^2}$.
    \item If $u_i \equiv 0 \pmod{p}$ and $v_j \equiv 0 \pmod{p}$, then $u_i^T v_j \equiv 0 \pmod{p^2}$ trivially.
    \item If exactly one of $u_i$, $v_j$ is $0 \pmod{p}$ (say $u_i \equiv 0 \pmod{p}$, $v_j \not\equiv 0 \pmod{p}$), then since $\operatorname{rank}_p(\bar{Q}_1) > 0$, there exists $u_k \not\equiv 0 \pmod{p}$, and $u_k$, $v_j$ are linearly dependent over $\mathbb{F}_p$. So $v_j = c u_k + p\beta$ for some $c \not\equiv 0 \pmod{p}$ and integer vector $\beta$. Thus,
    \[
        u_i^T v_j = c u_i^T u_k + p u_i^T \beta.
    \]
    Both terms are $0 \pmod{p^2}$, so $u_i^T v_j \equiv 0 \pmod{p^2}$.
\end{itemize}

Thus, $O = \ell(Q_1)\ell(Q_2) Q_1^T Q_2 \equiv 0 \pmod{p^2}$. Since $Q_1^T Q_2$ is rational, Lemma~\ref{lemma:level-mod} implies $\ell(Q_1^T Q_2) \mid \frac{\ell(Q_1)\ell(Q_2)}{p^2}$. As both levels are square-free, $p \nmid \frac{\ell(Q_1)\ell(Q_2)}{p^2}$, so $p \nmid \ell(Q_1^T Q_2)$.
\end{proof}

\begin{lemma}
\label{lemma:same-level}
Let $Q_1$ and $Q_2$ be primitive matrices with $\ell(Q_1) = \ell(Q_2) = \ell$, where $\ell$ is square-free. Let $\bar{Q}_1 = \ell Q_1$ and $\bar{Q}_2 = \ell Q_2$. If the column spaces of $\bar{Q}_1$ and $\bar{Q}_2$ coincide over $\mathbb{F}_p$ for every odd prime $p$ dividing $\ell$, then $Q_1^T Q_2$ is a permutation matrix.
\end{lemma}

\begin{proof}
Since $\ell(Q_1) = \ell(Q_2) = \ell$, we have $\ell(Q_1^T Q_2) \mid \ell^2$. By Lemma~\ref{lemma:p-out}, for each odd prime $p$ dividing $\ell$, $p \nmid \ell(Q_1^T Q_2)$. It follows that $\ell(Q_1^T Q_2) = 1$. Therefore, $Q_1^T Q_2$ is an integral orthogonal matrix, i.e., a permutation matrix.
\end{proof}

With these tools in hand, we can now prove Theorem~\ref{theorem:main}.

\section{Proof of Theorem \ref{theorem:main}}
\label{section:proof}

We now establish several lemmas that underpin the proof of Theorem~\ref{theorem:main}.

\begin{lemma}
\label{lemma:rank}
Let \( G \in \mathcal{F}_n \) and \( Q \in \Gamma(G) \). Then $Q$ is a primitive matrix.
\end{lemma}

\begin{proof}
Let $\bar{Q} = \ell(Q) Q$ and let $p$ be any prime dividing $\ell(Q)$. By Corollary~\ref{corollary:l-struct}, $\ell(Q)$ is odd and square-free. Since $\ell(Q) \mid d_n(W(G))$, $p$ also divides $d_n(W(G))$. By Theorem~\ref{theorem:q}, $Q^T W(G) = W(H)$ (which is integral) for some graph $H$ with the same generalized spectrum as $G$, so $\bar{Q}^T W(G) \equiv 0 \pmod{p}$. By the assumptions of Theorem~\ref{theorem:main}, $\operatorname{rank}_p W(G) = n-1$, so the nullspace of $W(G)$ over $\mathbb{F}_p$ is one-dimensional. Thus, $\operatorname{rank}_p(\bar{Q}) \le 1$.

On the other hand, since $p \mid \ell(Q)$, Lemma~\ref{lemma:level-mod} ensures that $\bar{Q} \not\equiv 0 \pmod{p}$, so $\operatorname{rank}_p(\bar{Q}) > 0$. Therefore, $\operatorname{rank}_p(\bar{Q}) = 1$ for every $p \mid \ell(Q)$, as required.
\end{proof}

The following corollary is immediate from Lemma~\ref{lemma:rank} and the fact that the nullspace of $W(G)$ over $\mathbb{F}_p$ is one-dimensional.

\begin{corollary}
\label{corollary:linear-basis}
Let $G \in \mathcal{F}_n$ and \( Q_1, Q_2 \in \Gamma(G) \). If a prime \( p \) divides both \( \ell(Q_1) \) and \( \ell(Q_2) \), then the columns of \( \ell(Q_1) Q_1 \) and \( \ell(Q_2) Q_2 \) span the same one-dimensional subspace over \( \mathbb{F}_p \).
\end{corollary}

With these results, we are ready to prove Theorem~\ref{theorem:main}.

\begin{proof}[\textbf{Proof of Theorem \ref{theorem:main}}]
Let \(H\) and \(F\) be two generalized cospectral mates of \(G\). Let \(Q_1, Q_2 \in \Gamma(G)\) such that
\[
A(H)= Q_1^TA(G)Q_1
\quad\text{and}\quad
A(F) = Q_2^TA(G)Q_2.
\]

We can write $A(G)$ in terms of $A(H)$ as:
\[
A(G) = Q_1 A(H) Q_1^T,
\]

Putting this into the expression for $A(F)$, we have:
\[
A(F) = Q_2^T Q_1 A(H) Q_1^T Q_2.
\]

Using Lemma~\ref{lemma:rank} and Corollary~\ref{corollary:linear-basis}, if \( \ell(Q_1) = \ell(Q_2) \), then the conditions for Lemma~\ref{lemma:same-level} are satisfied, which implies that \(Q_1^T Q_2\) is a permutation matrix. This means that \(H\) and \(F\) are isomorphic graphs, contradicting the assumption that they are distinct generalized cospectral mates. Therefore, each non-isomorphic generalized cospectral mate of \(G\) corresponds to a distinct value of the level \(\ell(Q)\) for some \(Q \in \Gamma(G)\), and no two generalized cospectral mates share the same level. This establishes a one-to-one correspondence between possible levels and non-isomorphic generalized cospectral mates.

For every matrix $Q \in \Gamma(G)$, Lemma~\ref{lemma:level-dn} ensures that $\ell(Q)$ divides $d_n(W(G))$, and Corollary~\ref{corollary:l-struct} guarantees that $\ell(Q)$ is odd and square-free. Lemma~\ref{lemma:pk-out} further restricts the possible odd prime divisors of $\ell(Q)$ to those $p$ for which $p^2$ divides $d_n(W(G))$. Thus, the level $\ell(Q)$ can only be a product of distinct odd primes $p_1, \ldots, p_k$ such that $p_i^2 \mid d_n(W(G))$ for each $i$.

There are $2^k$ possible square-free products of these $k$ primes, including $1$. The case $\ell(Q) = 1$ corresponds to a permutation matrix, which yields an isomorphic graph and does not produce a new cospectral mate. Therefore, the number of non-isomorphic generalized cospectral mates of $G$ is at most $2^k - 1$. This completes the proof.
\end{proof}

\section{Examples}
\label{section:examples}

\begin{figure*}[t!]
    \centering
    \begin{subfigure}[t]{0.5\textwidth}
        \centering
        \includegraphics[scale=0.3]{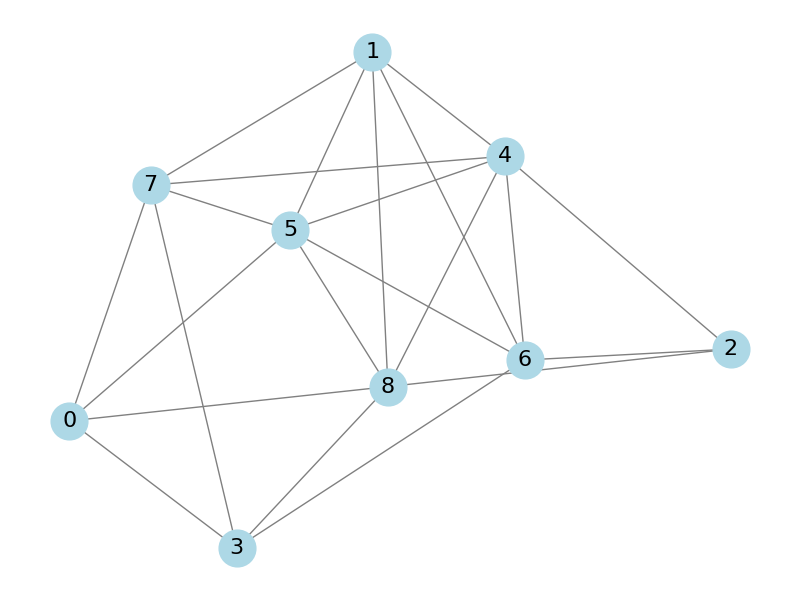}
        \caption{Graph $G$}
        \label{figure:e1-g}
    \end{subfigure}%
    ~ 
    \begin{subfigure}[t]{0.5\textwidth}
        \centering
        \includegraphics[scale=0.3]{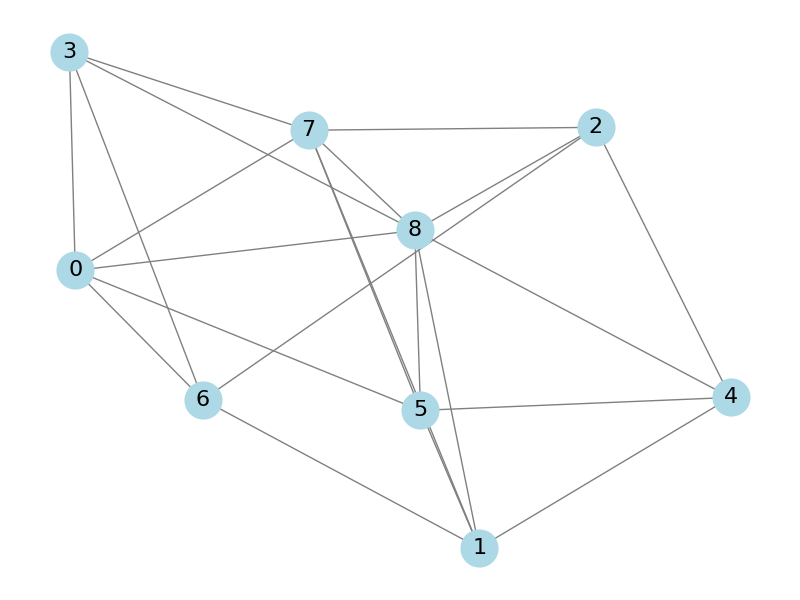}
        \caption{Graph $H$}
        \label{figure:e1-h}
    \end{subfigure}
    \caption{Example 1}
\end{figure*}

\begin{example}
Let $n=9$ and adjacency matrix of a graph $G$ (Figure~\ref{figure:e1-g}) be as follows:

\begin{equation*}
A(G) = \begin{pmatrix}
0 & 0 & 0 & 1 & 0 & 1 & 0 & 1 & 1 \\
0 & 0 & 0 & 0 & 1 & 1 & 1 & 1 & 1 \\
0 & 0 & 0 & 0 & 1 & 0 & 1 & 0 & 1 \\
1 & 0 & 0 & 0 & 0 & 0 & 1 & 1 & 1 \\
0 & 1 & 1 & 0 & 0 & 1 & 1 & 1 & 1 \\
1 & 1 & 0 & 0 & 1 & 0 & 1 & 1 & 1 \\
0 & 1 & 1 & 1 & 1 & 1 & 0 & 0 & 0 \\
1 & 1 & 0 & 1 & 1 & 1 & 0 & 0 & 0 \\
1 & 1 & 1 & 1 & 1 & 1 & 0 & 0 & 0
\end{pmatrix}
\end{equation*}

The determinant of the walk matrix of $G$ is
\[
\det{W(G)} = -1936 = (-1) \times 2^4 \times 11^2.
\]

The rank of $W(G)$ over $\mathbb{F}_{11}$ is $8$, which satisfies the conditions of Theorem~\ref{theorem:main}. Therefore, $G$ can have at most one non-isomorphic generalized cospectral mate. Indeed, there exists exactly one such mate, denoted by $H$ (see Figure~\ref{figure:e1-h}), whose adjacency matrix is given below.

\begin{equation*}
    A(H) =
    \begin{pmatrix}
0 & 0 & 0 & 1 & 0 & 1 & 1 & 1 & 1 \\
0 & 0 & 0 & 0 & 1 & 1 & 1 & 1 & 1 \\
0 & 0 & 0 & 0 & 1 & 0 & 1 & 1 & 1 \\
1 & 0 & 0 & 0 & 0 & 0 & 1 & 1 & 1 \\
0 & 1 & 1 & 0 & 0 & 1 & 0 & 0 & 1 \\
1 & 1 & 0 & 0 & 1 & 0 & 0 & 1 & 1 \\
1 & 1 & 1 & 1 & 0 & 0 & 0 & 0 & 0 \\
1 & 1 & 1 & 1 & 0 & 1 & 0 & 0 & 1 \\
1 & 1 & 1 & 1 & 1 & 1 & 0 & 1 & 0 \\
    \end{pmatrix}.
\end{equation*}.

\end{example}

\begin{figure*}[t!]
    \centering
    \begin{subfigure}[t]{0.5\textwidth}
        \centering
\includegraphics[scale=0.3]{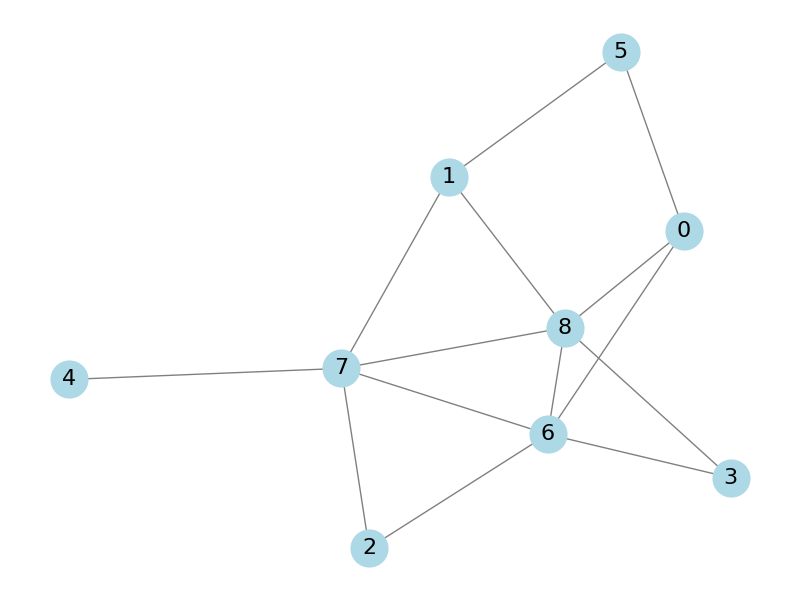}
\caption{Graph $N$}
\label{figure:e2-n}
    \end{subfigure}%
    ~ 
    \begin{subfigure}[t]{0.5\textwidth}
        \centering
\includegraphics[scale=0.3]{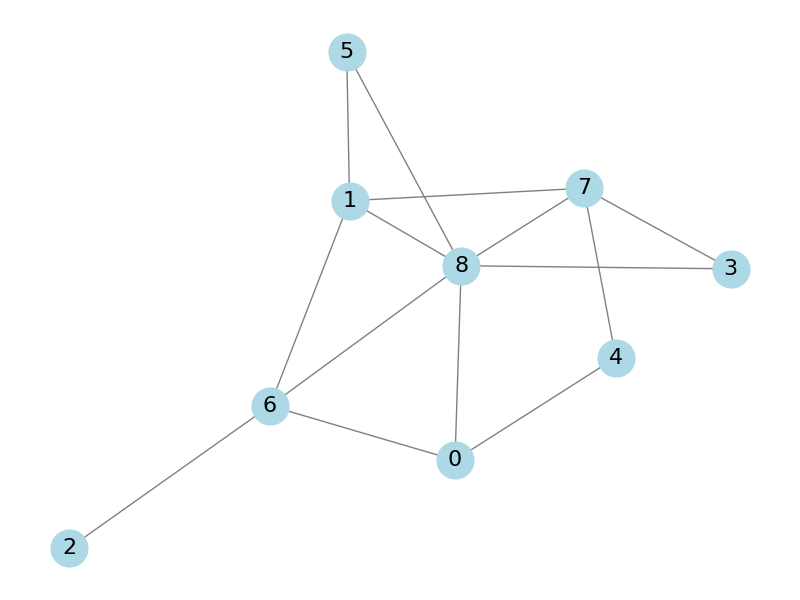}
\caption{Graph $M$}
\label{figure:e2-m}
    \end{subfigure}
    \caption{Example 2}
\end{figure*}

\begin{example}
Let $n=9$ and adjacency matrix of a graph $N$ (Figure~\ref{figure:e2-n}) be as follows:

\begin{equation*}
A(N) = \begin{pmatrix}
0 & 0 & 0 & 0 & 0 & 1 & 1 & 0 & 1 \\
0 & 0 & 0 & 0 & 0 & 1 & 0 & 1 & 1 \\
0 & 0 & 0 & 0 & 0 & 0 & 1 & 1 & 0 \\
0 & 0 & 0 & 0 & 0 & 0 & 1 & 0 & 1 \\
0 & 0 & 0 & 0 & 0 & 0 & 0 & 1 & 0 \\
1 & 1 & 0 & 0 & 0 & 0 & 0 & 0 & 0 \\
1 & 0 & 1 & 1 & 0 & 0 & 0 & 1 & 1 \\
0 & 1 & 1 & 0 & 1 & 0 & 1 & 0 & 1 \\
1 & 1 & 0 & 1 & 0 & 0 & 1 & 1 & 0 \\
\end{pmatrix}
\end{equation*}

The determinant of the walk matrix of $N$ is
\[
\det{W(N)} = 10224 = 2^4 \times 3^2 \times71.
\]

We compute that $\operatorname{rank}_{3}{W(N)} = 8$, which satisfies the conditions of Theorem~\ref{theorem:main}. Therefore, $N$ can have at most one non-isomorphic generalized cospectral mate. Indeed, there exists exactly one such mate, denoted by $M$ (see Figure~\ref{figure:e2-m}), whose adjacency matrix is given below.

\begin{equation*}
    A(M) =
    \begin{pmatrix}
0 & 0 & 0 & 1 & 0 & 1 & 1 & 1 & 1 \\
0 & 0 & 0 & 0 & 1 & 1 & 1 & 1 & 1 \\
0 & 0 & 0 & 0 & 1 & 0 & 1 & 1 & 1 \\
1 & 0 & 0 & 0 & 0 & 0 & 1 & 1 & 1 \\
0 & 1 & 1 & 0 & 0 & 1 & 0 & 0 & 1 \\
1 & 1 & 0 & 0 & 1 & 0 & 0 & 1 & 1 \\
1 & 1 & 1 & 1 & 0 & 0 & 0 & 0 & 0 \\
1 & 1 & 1 & 1 & 0 & 1 & 0 & 0 & 1 \\
1 & 1 & 1 & 1 & 1 & 1 & 0 & 1 & 0 \\
    \end{pmatrix}.
\end{equation*}

It is worth noting that, although $\det{W(N)} = \det{W(M)}$, we have $\operatorname{rank}_{3} W(M) = 7$, which does not satisfy the conditions of Theorem~\ref{theorem:main}. Therefore, we cannot use the theorem directly for the graph $M$ to get an upper bound.

\end{example}

\section{Conclusion}
\label{section:conclusion}
As discussed earlier, the walk matrix of a graph is precisely the controllability matrix of a linear dynamical system defined on the graph, where the adjacency matrix $A$ encodes the topology and the all-ones vector $e$ serves as the input vector:
\begin{equation}
    \dot{x}(t) = A x(t) + e u(t).
\end{equation}
Here, the controllability matrix
\[
\mathcal{C} = [e,\, Ae,\, A^2 e,\, \ldots,\, A^{n-1}e]
\]
is exactly the walk matrix $W(G)$. The system is controllable if and only if $W(G)$ is non-singular. This establishes a direct link between spectral graph theory and control theory: the walk matrix-based criteria for spectral characterization are fundamentally controllability conditions for the associated dynamical system.

More generally, one can consider
\begin{equation}
    \dot{x}(t) = A x(t) + b u(t),
\end{equation}
where $b$ is any input vector. The corresponding controllability matrix is $[b,\, Ab,\, \ldots,\, A^{n-1}b]$. While most spectral characterizations use $e$ as the input, recent work by Qiu et al.~\cite{qiu2023generalized} has shown that using different input vectors $b$ can be crucial, especially for regular graphs where $e$ may not guarantee controllability.

In this work, we established new upper bounds on the number of generalized cospectral mates of a graph, based on the arithmetic properties of the walk matrix determinant, and extended previous results to a broader family of graphs. Our results highlight the deep connection between spectral characterization and controllability: the arithmetic structure of the controllability (walk) matrix plays a crucial role in determining whether a graph is DGS.

Exploring alternative input vectors in the controllability matrix framework—beyond the all-ones vector—remains a promising direction for future research. This could lead to new families of graphs determined by their (generalized) spectra and further strengthen the interplay between spectral graph theory and control theory.

\bibliographystyle{plain}
\bibliography{root}

\begin{thebibliography}{10}

\bibitem{fisher1966hearing}
Michael~E Fisher.
\newblock On hearing the shape of a drum.
\newblock {\em Journal of Combinatorial Theory}, 1(1):105--125, 1966.

\bibitem{godsil1982constructing}
Chris~D Godsil and Brendan~D McKay.
\newblock Constructing cospectral graphs.
\newblock {\em Aequationes Mathematicae}, 25:257--268, 1982.

\bibitem{gunthard1956zusammenhang}
Hs~H G{\"u}nthard and Hans Primas.
\newblock Zusammenhang von graphentheorie und mo-theorie von molekeln mit systemen konjugierter bindungen.
\newblock {\em Helvetica Chimica Acta}, 39(6):1645--1653, 1956.

\bibitem{kac1966can}
Mark Kac.
\newblock Can one hear the shape of a drum?
\newblock {\em The american mathematical monthly}, 73(4P2):1--23, 1966.

\bibitem{qiu2023generalized}
Lihong Qiu, Yizhe Ji, Lihuan Mao, and Wei Wang.
\newblock Generalized spectral characterizations of regular graphs based on graph-vectors.
\newblock {\em Linear Algebra and its Applications}, 663:116--141, 2023.

\bibitem{qiu2021oriented}
Lihong Qiu, Wei Wang, and Wei Wang.
\newblock Oriented graphs determined by their generalized skew spectrum.
\newblock {\em Linear Algebra and its Applications}, 622:316--332, 2021.

\bibitem{qiu2023smith}
Lihong Qiu, Wei Wang, and Hao Zhang.
\newblock Smith normal form and the generalized spectral characterization of graphs.
\newblock {\em Discrete Mathematics}, 346(1):113177, 2023.

\bibitem{raza2025upper}
Muhammad Raza, Obaid~Ullah Ahmed, Mudassir Shabbir, Xenofon Koutsoukos, and Waseem Abbas.
\newblock An upper bound on generalized cospectral mates of oriented graphs using skew-walk matrices.
\newblock {\em arXiv preprint arXiv:2504.17278}, 2025.

\bibitem{schwenk1973almost}
Allen~J Schwenk.
\newblock Almost all trees are cospectral.
\newblock {\em New directions in the theory of graphs}, pages 275--307, 1973.

\bibitem{wang2013generalized}
Wei Wang.
\newblock Generalized spectral characterization of graphs: Revisited.
\newblock {\em arXiv preprint arXiv:1309.6090}, 2013.

\bibitem{wang2017simple}
Wei Wang.
\newblock A simple arithmetic criterion for graphs being determined by their generalized spectra.
\newblock {\em Journal of Combinatorial Theory, Series B}, 122:438--451, 2017.

\bibitem{wang2006excluding}
Wei Wang and Cheng-Xian Xu.
\newblock An excluding algorithm for testing whether a family of graphs are determined by their generalized spectra.
\newblock {\em Linear algebra and its Applications}, 418(1):62--74, 2006.

\bibitem{wang2006sufficient}
Wei Wang and Cheng-xian Xu.
\newblock A sufficient condition for a family of graphs being determined by their generalized spectra.
\newblock {\em European Journal of Combinatorics}, 27(6):826--840, 2006.

\bibitem{wang2023graphs}
Wei Wang and Tao Yu.
\newblock Graphs with at most one generalized cospectral mate.
\newblock {\em The Electronic Journal of Combinatorics}, pages P1--38, 2023.

\end{thebibliography}

\end{document}